\documentclass[12pt]{article}
\usepackage{amsthm,amsfonts,amssymb,amsmath}
\usepackage{stmaryrd}
\usepackage{cite,hyperref}
\usepackage{epsfig}
\usepackage{url}
\usepackage{xcolor,tikz}
\usetikzlibrary{positioning, arrows.meta,calc}
\usetikzlibrary{matrix}
\usepackage{multicol,graphicx}
\usepackage{fullpage}
\usepackage{bbm}
\usetikzlibrary{decorations}
\usepackage{svg}

\usepackage[shortlabels]{enumitem}

\numberwithin{equation}{section}

\newtheorem{thm}[equation]{Theorem}

\newtheorem{lem}[equation]{Lemma}

\theoremstyle{definition}

\newtheorem*{ai}{Acknowledgements and AI Declaration}

\theoremstyle{remark}

\title{Orthogonal Pairs in Maps from the Sphere to the Circle}
\author{Jonathan A. Noel\thanks{Department of Mathematics and Statistics, University of Victoria, Victoria, B.C., Canada. E-mail: {\tt noelj@uvic.ca}. Research supported by NSERC Discovery Grant RGPIN-2021-02460.}}

\DeclareTextCompositeCommand{\v}{OT1}{l}{l\nobreak\hspace{-.1em}'}
\DeclareTextCompositeCommand{\v}{OT1}{t}{t\nobreak\hspace{-.1em}'\nobreak\hspace{-.15em}}

\begin{document}

\maketitle

\begin{abstract}
We prove that, for any $f:S^2\to S^1$ and any $\varepsilon>0$, there exist orthogonal vectors $x,y\in S^2$ such that the length of the shortest arc between $f(x)$ and $f(y)$ is at most $\pi/2 +\varepsilon$. This proves a conjecture of Ghebleh from 2007 that the circular chromatic number of the real orthogonality graph is equal to four. 
\end{abstract}

\section{Introduction}

For $n\geq1$, let $S^n$ denote the $n$-dimensional unit sphere, i.e.
\[S^n:=\left\{x\in \mathbb{R}^{n+1}: \|x\|_2=1\right\}.\]
Our main result is the following. We remark that, in this theorem, the function $f$ is arbitrary, with no assumptions on continuity, measurability, or anything else. 

\begin{thm}
\label{th:spheres}
Let $f:S^2\to S^1$. Then, for any $\varepsilon>0$, there exist orthogonal $x$ and $y$  in $S^2$ such that the length of the shortest arc between $f(x)$ and $f(y)$ is at most $\pi/2 +\varepsilon$.
\end{thm}

This work is motivated by a conjecture of Ghebleh~\cite{Ghebleh07} in graph theory. Let $\mathcal{O}$ be the graph with vertex set $S^2$ where $x$ and $y$ are adjacent if they are orthogonal.\footnote{In~\cite{Ghebleh07,DeVos+09+}, the vertex set of $\mathcal{O}$ is taken to be the set of lines through the origin in $\mathbb{R}^3$; however, for our purposes, it is equivalent to deal with $S^2$.} The Kochen--Specker Theorem~\cite{KochenSpecker67} in quantum mechanics implies that there is no independent set in $\mathcal{O}$ that meets every triangle. As a consequence, the chromatic number of $\mathcal{O}$ is at least four. In fact, it is exactly equal to four. Godsil and Zaks~\cite{GodsilZaks88} exhibited an explicit $4$-colouring obtained by placing a regular octahedron at the origin and assigning a colour to each vector $x\in S^2$ based on the pair of opposite faces the line spanned by $x$ passes through.\footnote{One needs to be a little careful about the boundaries between faces, but this can be done.} 

Given a non-empty graph $G$, the \emph{circular chromatic number} of $G$, denoted $\chi_c(G)$, is the infimum over all $r\geq1$ such that there exists a mapping $f:V(G)\to S^1$ such that if $uv\in E(G)$, then the length of the shortest arc between $f(u)$ and $f(v)$ is at least $2\pi/r$. The circular chromatic number was first studied, in a different form and under a different name, by Vince~\cite{Vince88}. The definition given here is closer to the one given by Zhu~\cite[Section~2]{Zhu92}; see also the survey~\cite{Zhu01}.  

A fundamental result on the circular chromatic number is that $\left\lceil\chi_c(G)\right\rceil=\chi(G)$ for every non-empty graph $G$ of finite chromatic number. This was proven for finite graphs by Vince~\cite[Theorem~4]{Vince88} and extends to infinite graphs via a standard compactness argument; see~\cite[pp.~372--373]{Zhu01}. Therefore, the circular chromatic number is a refinement of the usual chromatic number. For the graph $\mathcal{O}$, this tells us that $3<\chi_c(\mathcal{O})\leq 4$. A lower bound of $7/2$ appears in~\cite{DeVos+09+} and Ghebleh~\cite[Theorem~4.40]{Ghebleh07} proved $\chi_c(\mathcal{O})\geq 11/3$ and conjectured~\cite[Conjecture~4.41]{Ghebleh07} that $\chi_c(\mathcal{O})=4$; see also~\cite{OPG}. The following theorem confirms this conjecture. It is easily seen to be equivalent to Theorem~\ref{th:spheres}.

\begin{thm}
\label{th:chi_c}
$\chi_c(\mathcal{O})=4$. 
\end{thm}

In the next section, we reduce the proof of Theorem~\ref{th:spheres} to the statement that the graph-theoretic neighbourhood of every somewhere dense subset of $S^2$ contains an odd cycle (Theorem~\ref{th:nbhd}). We then prove Theorem~\ref{th:nbhd} in Section~\ref{sec:proof}. The proof uses a construction of odd cycles near the equator of $S^2$, which is reminiscent of Lov\'asz's umbrella construction~\cite{Lovasz79} and subsequent constructions in~\cite{DeCortePikhurko16,CherkashinVoronov24}. The key idea, which is adapted from~\cite{CherkashinVoronov24}, is to use the Implicit Function Theorem to perturb such an odd cycle so that all of its points lie in the neighbourhood of an arbitrary somewhere dense set. 

\section{A Reduction}
\label{sec:prelim}

We start with a standard equivalent definition of the circular chromatic number. Given positive integers $p$ and $q$ with $p\geq q$, define a \emph{$(p,q)$-colouring} of a graph $G$ to be a map $f:V(G)\to \{0,1,\dots,p-1\}$ such that if $uv\in E(G)$, then $q\leq |f(u)-f(v)|\leq p-q$. The following equivalent definition of the circular chromatic number was proven in the case of finite graphs by Zhu~\cite[Theorem~1]{Zhu92}. It is also true for infinite graphs; the argument is essentially the same as the one used to prove~\cite[Lemma~4.4]{deBoerBuysZuiddam24}.

\begin{lem}[See~\cite{Zhu92,deBoerBuysZuiddam24}]
\label{lem:finite}
For any non-empty graph $G$,
\[\chi_c(G)=\inf\left\{p/q: G\text{ has a }(p,q)\text{-colouring}\right\}.\]
\end{lem}

As usual, we view the topology of $S^2$ as being generated by open spherical caps. For a topological space $(X,\tau)$, a subset $S$ of $X$ is said to be \emph{nowhere dense} if the interior of its closure is empty and \emph{somewhere dense} otherwise. The next lemma is used to show that, under any $(p,q)$-colouring of $\mathcal{O}$, at least one of the colour classes must be somewhere dense. 

\begin{lem}[See, e.g.,~{\cite[Problem~6.42]{Viro+08}}]
\label{lem:nowhereDense}
Let $(X,\tau)$ be a topological space. A union of finitely many nowhere dense subsets of $X$ is nowhere dense.
\end{lem}

Next, we deduce Theorem~\ref{th:chi_c} (and, therefore, Theorem~\ref{th:spheres}) from the following theorem. After this, the rest of the paper will be devoted to proving this theorem. Given a vertex $x\in S^2$, we let 
\[N(x):=\left\{y\in S^2: xy\in E(\mathcal{O})\right\}.\]
Also, for a set $X\subseteq S^2$, we let $N(X):=\bigcup_{x\in X} N(x)$ and note that $X$ and $N(X)$ may have elements in common. 

\begin{thm}
\label{th:nbhd}
If $X$ is a somewhere dense subset of $S^2$, then the subgraph of $\mathcal{O}$ induced by $N(X)$ contains an odd cycle. 
\end{thm}

\begin{proof}[Proof of Theorem~\ref{th:chi_c}]
Suppose, to the contrary, that $\chi_c(\mathcal{O})<4$. Then, by Lemma~\ref{lem:finite}, there exist positive integers $p$ and $q$ such that $q\leq p<4q$ and $\mathcal{O}$ has a $(p,q)$-colouring, say $f:S^2\to \{0,\dots,p-1\}$. Note that $p\geq3q$ because $\mathcal{O}$ contains a triangle; e.g. on the vectors $(1,0,0),(0,1,0)$ and $(0,0,1)$. 

By Lemma~\ref{lem:nowhereDense}, there must exist $c\in\{0,\dots,p-1\}$ such that $f^{-1}(c)$ is somewhere dense. Without loss of generality $c=0$. By Theorem~\ref{th:nbhd}, there is an odd cycle $C$ in $\mathcal{O}$ in which all of the vertices are in $N(f^{-1}(0))$. By definition of a  $(p,q)$-colouring, we must have that all vertices of $C$ are coloured from $A:=\{q,\dots,2q-1\}$ or $B:=\{2q,\dots, p-q\}$. Note that $A$ consists of $q$ consecutive colours, and $B$ consists of $p-3q+1$ consecutive colours; since $3q\leq p\leq 4q-1$, we have $1\leq|B|\leq q$. Therefore, no two adjacent vertices can both be coloured from $A$ and no two adjacent vertices can both be coloured from $B$. However, this contradicts the fact that $C$ does not have a proper $2$-colouring and completes the proof. 
\end{proof}

\section{Neighbourhoods of Somewhere Dense Sets}
\label{sec:proof}

The proof of Theorem~\ref{th:nbhd} applies the Implicit Function Theorem in the following form; a similar application of the Implicit Function Theorem can be found in~\cite{CherkashinVoronov24}.

\begin{thm}[See, e.g.,~{\cite[Theorem~9.28]{Rudin76}}]
\label{th:IFT}
Let $F:\mathbb{R}^{p+q}\to \mathbb{R}^q$ be a continuously differentiable function. We view $F(x,y)$ as a function of variables $x=(x_1,\dots,x_p)\in\mathbb{R}^p$ and $y=(y_1,\dots,y_q)\in \mathbb{R}^q$ and write $F(x,y)=(F_1(x,y),\dots,F_q(x,y))$ where $F_i:\mathbb{R}^{p+q}\to\mathbb{R}$ for all $1\leq i\leq q$. Suppose that $a\in\mathbb{R}^p$ and $b\in \mathbb{R}^q$ satisfy $F(a,b)=0\in \mathbb{R}^q$. If the \emph{Jacobian matrix}
\[
D_yF(a,b):=
\begin{bmatrix}
\frac{\partial F_1}{\partial y_1}(a,b)
& \cdots &
\frac{\partial F_1}{\partial y_q}(a,b)
\\
\vdots & \ddots & \vdots
\\
\frac{\partial F_q}{\partial y_1}(a,b)
& \cdots &
\frac{\partial F_q}{\partial y_q}(a,b)
\end{bmatrix}
\]
is invertible, then there exists an open set $W\subseteq \mathbb{R}^p$ with $a\in W$ and a continuously differentiable function $g:W\to \mathbb{R}^q$ such that $g(a)=b$ and $F(x,g(x))=0$ for all $x\in W$. 
\end{thm}

We now present the proof of Theorem~\ref{th:nbhd}.

\begin{proof}[Proof of Theorem~\ref{th:nbhd}]
Let $X\subseteq S^2$ be somewhere dense and let $U$ be an open spherical cap contained in the interior of the closure of $X$. Without loss of generality, we may assume that $U$ is centred on the vector $(0,0,1)$. That is, we can assume that 
\[U=\{(u_1,u_2,u_3)\in S^2: u_3>1-\delta\}\]
for some $\delta>0$. Also, by definition of $U$, we have that $X$ is dense in $U$. 

Let $k$ be a positive integer, which will be chosen large with respect to $\delta$, and let $n=4k+3$. In particular, $n$ is odd. The idea of the proof is as follows. The first step is to select $n$ points $a_0,\dots,a_{n-1}$ in $U$ and points $b_0,\dots,b_{n-1}$ near the equator such that $b_i$ is orthogonal to $a_i$ and $b_{i+1}$ for all $0\leq i\leq n-1$, where indices are viewed modulo $n$. That is, $b_0,\dots,b_{n-1}$ form an odd closed walk that is dominated by $a_0,\dots,a_{n-1}$. After that, the idea is to show that this choice is sufficiently flexible that we can vary the points $a_0,\dots,a_{n-1}$ freely within an open set and, for every such choice, say $x_0,\dots,x_{n-1}$, we can find a perturbation $w_0,\dots,w_{n-1}$ of $b_0,\dots,b_{n-1}$ which forms an odd closed walk and is dominated by $x_0,\dots,x_{n-1}$. The Implicit Function Theorem (Theorem~\ref{th:IFT}) will allow us to do this. Since $X$ is dense in $U$, we can choose $x_0,\dots,x_{n-1}$ so that they are contained in $X$, which will complete the proof. 

Let us now give the details. We start by constructing $a_0,\dots,a_{n-1},b_0,\dots,b_{n-1}$; see Figure~\ref{fig:ab} for a diagram. For $0\leq i\leq n$, define
\[v_i:=\left(\cos\left(\frac{i(n+1)\pi}{2n}\right),\sin\left(\frac{i(n+1)\pi}{2n}\right),\sqrt{\sin\left(\frac{\pi}{2n}\right)}\right).\]
Note that $v_n=v_0$ because $n=4k+3$ and so $n+1$ is a multiple of four. Also, for all $0\leq i\leq n-1$, we have
\begin{equation}
\label{eq:vnorm}
\|v_i\|_2=\sqrt{\cos^2\left(\frac{i(n+1)\pi}{2n}\right)+\sin^2\left(\frac{i(n+1)\pi}{2n}\right)+\sin\left(\frac{\pi}{2n}\right)}=\sqrt{1+\sin\left(\frac{\pi}{2n}\right)}.\end{equation}
We let $b_i=\left(1+\sin\left(\frac{\pi}{2n}\right)\right)^{-1/2}\cdot v_i$ so that $b_i\in S^2$. Note that, if $n$ is large, then the third coordinate of $b_i$ is close to zero. Thus, the vectors $b_0,\dots,b_{n-1}$ are close to the equator. 

Now, let us show that $b_0,\dots,b_{n-1}$ form a closed walk in $\mathcal{O}$. For $0\leq i\leq n-1$, we have 
\begin{align*}
\langle v_i,v_{i+1}\rangle =& \cos\left(\frac{i(n+1)\pi}{2n}\right)\cos\left(\frac{(i+1)(n+1)\pi}{2n}\right)\\
&+ \sin\left(\frac{i(n+1)\pi}{2n}\right)\sin\left(\frac{(i+1)(n+1)\pi}{2n}\right) + \sin\left(\frac{\pi}{2n}\right).
\end{align*}
Using the identity $\cos(A-B)=\cos(A)\cos(B) + \sin(A)\sin(B)$, the expression on the right side becomes
\[\cos\left(\frac{(n+1)\pi}{2n}\right) + \sin\left(\frac{\pi}{2n}\right)=\cos\left(\frac{\pi}{2}+\frac{\pi}{2n}\right)+ \sin\left(\frac{\pi}{2n}\right) = -\sin\left(\frac{\pi}{2n}\right)+\sin\left(\frac{\pi}{2n}\right)=0.\]
Therefore, they are orthogonal. Recall that $v_n=v_0$. Thus, $b_0,\dots,b_{n-1}$ form an odd closed walk in $\mathcal{O}$. 

Next, define
\[u_0:=\left(-1,1,\frac{1}{\sqrt{\sin\left(\frac{\pi}{2n}\right)}}\right)\]
and, for $1\leq i\leq n-1$, let
\[u_i:=\left(-\cos\left(\frac{i(n+1)\pi}{2n}\right),-\sin\left(\frac{i(n+1)\pi}{2n}\right),\frac{1}{\sqrt{\sin\left(\frac{\pi}{2n}\right)}}\right).\]
We choose $u_0$ differently from $u_1,\dots,u_{n-1}$ for technical reasons; in particular, it will be helpful for showing that the Jacobian is invertible in our application of the Implicit Function Theorem. Note that
\[\|u_0\|_2=\sqrt{2+\frac{1}{\sin\left(\frac{\pi}{2n}\right)}}\]
and, for $1\leq i\leq n-1$,
\[\|u_i\|_2 = \sqrt{1+\frac{1}{\sin\left(\frac{\pi}{2n}\right)}}.\]
Define $a_i:=\|u_i\|_2^{-1}\cdot u_i$ for all $0\leq i\leq n-1$. We choose $n$ large enough so that 
\[\frac{1}{\sqrt{2+\frac{1}{\sin\left(\frac{\pi}{2n}\right)}}\cdot \sqrt{\sin\left(\frac{\pi}{2n}\right)}}=\frac{1}{\sqrt{2\sin\left(\frac{\pi}{2n}\right) + 1}}>1-\delta\]
which ensures that all of $a_0,a_1,\dots, a_{n-1}$ are in $U$. We observe that
\[\langle u_0,v_0\rangle = \left\langle\left(-1,1,\frac{1}{\sqrt{\sin\left(\frac{\pi}{2n}\right)}}\right), \left(1,0,\sqrt{\sin\left(\frac{\pi}{2n}\right)}\right)\right\rangle=0.\]
Also, for $1\leq i\leq n-1$, using the identity $\cos^2\theta + \sin^2\theta=1$, we get that $\langle u_i,v_i\rangle=0$. Thus, $a_i$ is orthogonal to $b_i$ for all $0\leq i\leq n-1$.

The goal in the rest of the proof is to use the Implicit Function Theorem to show that we can perturb the points $a_0,\dots,a_{n-1}$ within an open set so that, for every such perturbation, we can find an odd closed walk in the neighbourhood. First, for $0\leq i\leq n-1$, denote the coordinates of $a_i$ by $a_i=(a_{i,1},a_{i,2},a_{i,3})$ and define
\[a:=\left(a_{0,1},a_{0,2},a_{0,3}, a_{1,1},a_{1,2},a_{1,3},\dots,a_{n-1,1},a_{n-1,2},a_{n-1,3}\right).\]
We let 
\[x=(x_{0,1},x_{0,2},x_{0,3}, x_{1,1},x_{1,2},x_{1,3},\dots,x_{n-1,1},x_{n-1,2},x_{n-1,3})\in \mathbb{R}^{3n}\]
be a vector of $3n$ real variables where, for $0\leq i\leq n-1$, we let $x_i=(x_{i,1},x_{i,2},x_{i,3})$. Next, define
\[h_0:= \left(\sin\left(\frac{\pi}{2n}\right), 1+\sin\left(\frac{\pi}{2n}\right),-\sqrt{\sin\left(\frac{\pi}{2n}\right)}\right)\]
and, for $1\leq i\leq n-1$, define
\[h_i:= \left(-\sin\left(\frac{i(n+1)\pi}{2n}\right), \cos\left(\frac{i(n+1)\pi}{2n}\right),0\right).\]
We observe that $h_i$ is orthogonal to both $a_i$ and $b_i$ for all $0\leq i\leq n-1$. Let 
\[y=\left(y_{0,1},y_{1,1},\dots,y_{n-1,1},y_{0,2},y_{1,2},\dots,y_{n-1,2}\right)\in \mathbb{R}^{2n}\]
be a vector of $2n$ real variables. For every choice of $y\in \mathbb{R}^{2n}$ and $0\leq i\leq n-1$, let 
\[z_i(y) = b_i + y_{i,1}a_i + y_{i,2}h_i.\]
We define a function $F:\mathbb{R}^{3n+2n}\to\mathbb{R}^{2n}$ by 
\[F(x,y):=(A_0(x,y),\dots, A_{n-1}(x,y),B_0(x,y),\dots,B_{n-1}(x,y)),\] where $x\in \mathbb{R}^{3n}$ and $y\in \mathbb{R}^{2n}$. For $0\leq i\leq n-1$, define
\[A_i(x,y):=\left\langle x_i,z_i(y) \right\rangle,\]
\[B_i(x,y):=\left\langle z_i(y),z_{i+1}(y) \right\rangle\]
where the indices are viewed modulo $n$. By construction, each coordinate of $F$ is a polynomial and, therefore, $F$ is continuously differentiable. Also, $F(a,0)=0$, where $0$ denotes the zero vector of $\mathbb{R}^{2n}$. 

Let us now compute the entries of the relevant Jacobian matrix. For any $x$ and $y$ and $0\leq i,j\leq n-1$, we have
\[\frac{\partial A_i}{\partial y_{j,1}}(x,y) = \begin{cases}\langle x_i,a_i\rangle &\text{if }j=i,\\0&\text{otherwise},\end{cases} \quad\quad\quad\frac{\partial A_i}{\partial y_{j,2}}(x,y) = \begin{cases}\langle x_i,h_i\rangle &\text{if }j=i,\\0&\text{otherwise},\end{cases}\]
\[\frac{\partial B_i}{\partial y_{j,1}}(x,y) = \begin{cases}\langle a_i,z_{i+1}(y)\rangle &\text{if }j=i,\\ \langle z_i(y),a_{i+1}\rangle&\text{if }j=i+1,\\0&\text{otherwise},\end{cases} \quad\quad\quad\frac{\partial B_i}{\partial y_{j,2}}(x,y) = \begin{cases}\langle h_i,z_{i+1}(y)\rangle &\text{if }j=i,\\\langle z_i(y),h_{i+1}\rangle & \text{if }j=i+1,\\0&\text{otherwise}.\end{cases}\]
So, if we set $x=a$ and $y=0$, where $0$ is the zero vector in $\mathbb{R}^{2n}$, we get 
\[\frac{\partial A_i}{\partial y_{j,1}}(a,0) = \begin{cases}1 &\text{if }j=i,\\0&\text{otherwise},\end{cases} \quad\quad\quad\frac{\partial A_i}{\partial y_{j,2}}(a,0) = 0\]
\[
\frac{\partial B_i}{\partial y_{j,1}}(a,0) = \begin{cases}\langle a_i,b_{i+1}\rangle &\text{if }j=i,\\ \langle b_i,a_{i+1}\rangle&\text{if }j=i+1,\\0&\text{otherwise},\end{cases} \quad\quad\quad\frac{\partial B_i}{\partial y_{j,2}}(a,0) = \begin{cases}\langle h_i,b_{i+1}\rangle &\text{if }j=i,\\\langle b_i,h_{i+1}\rangle & \text{if }j=i+1,\\0&\text{otherwise}.\end{cases}\]
Thus, the Jacobian has the form
\[D_yF(a,0) = \begin{bmatrix}I & 0\\ L & M\end{bmatrix}\]
where $I$ is the $n\times n$ identity matrix, $0$ is the $n\times n$ all-zero matrix, and both $L$ and $M$ are $n\times n$ matrices. Therefore, $D_yF(a,0)$ is invertible if and only if $M$ is. We have
\[M_{i,j}=\frac{\partial B_i}{\partial y_{j,2}}(a,0)=\begin{cases}\langle h_i,b_{i+1}\rangle &\text{if }j=i,\\\langle b_i,h_{i+1}\rangle & \text{if }j=i+1,\\0&\text{otherwise}.\end{cases}\]
Let $s:=\sin\left(\frac{\pi}{2n}\right)$ and $c:=\cos\left(\frac{\pi}{2n}\right)$. By \eqref{eq:vnorm} and the definitions of $h_i$, $b_i$ and $v_i$, we have
\[\sqrt{1+s}\cdot M = \begin{bmatrix}
\left(1+s\right)\left(c-s\right) & -c & 0 & 0& \cdots & 0\\
0 & c & -c & 0&\cdots & 0\\
0 & 0 & c & -c&\cdots & 0\\
\vdots & \vdots&  &\ddots & & \vdots\\
0 &0&0 &\cdots&c&-c\\ 
-\left(1+s\right)\left(c+s\right)&0&0&\cdots & 0& c
\end{bmatrix}.\]
Now, let $N$ be the matrix obtained from $\sqrt{1+s}\cdot M$ by replacing the first row with the sum of all of the rows. Then $M$ is invertible if and only if $N$ is invertible. Using $\left(1+s\right)\left(c-s\right)-\left(1+s\right)\left(c+s\right)=-2s(s+1)$ and the cancellation between $c$ and $-c$, we have 
\[N= \begin{bmatrix}
-2s(s+1) & 0 & 0 & 0& \cdots & 0\\
0 & c & -c & 0&\cdots & 0\\
0 & 0 & c & -c&\cdots & 0\\
\vdots & \vdots&  &\ddots & & \vdots\\
0 &0&0 &\cdots&c&-c\\ 
-\left(1+s\right)\left(c+s\right)&0&0&\cdots & 0& c
\end{bmatrix}.\]
We compute the determinant of $N$ by expanding along the first row. The matrix obtained from $N$ by deleting the first row and column is upper triangular with all diagonal entries equal to $c$; thus, it has determinant $c^{n-1}$. Therefore, $\det(N)= -2s(s+1)c^{n-1}\neq 0$. Thus, $D_yF(a,0)$ is invertible. 

So, by the Implicit Function Theorem, there exists an open set $W\subseteq \mathbb{R}^{3n}$ with $a\in W$ and a continuously differentiable function $g:W\to\mathbb{R}^{2n}$ such that $g(a)=0$ and $F(x,g(x))=0$ for all $x\in W$. Since $a_0,\dots,a_{n-1}\in U$ and $X$ is dense in $U$, we may choose $x\in W$ such that $x_0,\dots,x_{n-1}\in X\cap U$. Let $y=g(x)$ and consider the vectors $z_0(y),\dots,z_{n-1}(y)$. Since $\langle b_i,z_i(y)\rangle=1$ by construction of $z_i(y)$, we have that $z_i(y)$ is not the zero vector. So, we can define $w_i:=\|z_i(y)\|_2^{-1}z_i(y)$ for $0\leq i\leq n-1$. Thus, $w_0,\dots,w_{n-1}$ form an odd closed walk in $\mathcal{O}$ where $x_i$ is adjacent to $w_i$ for all $0\leq i\leq n-1$. Every odd closed walk contains an odd cycle, and so the proof is complete.
\end{proof}

\begin{figure}[htbp]
    \centering
    \includegraphics[width=0.75\linewidth]{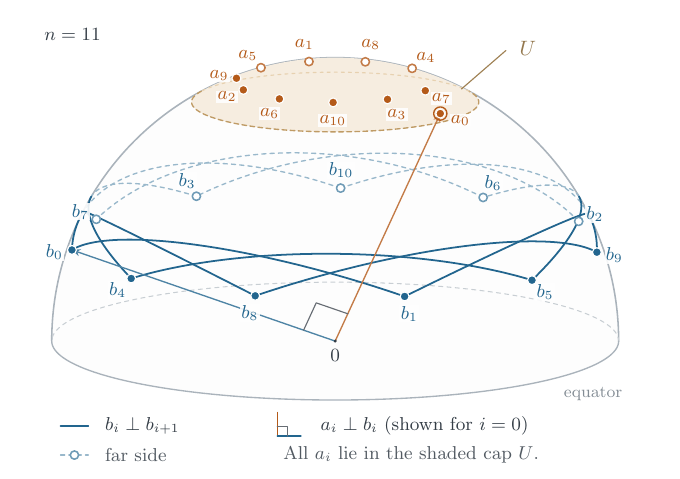}
    \caption{A depiction of the construction of $a_0,\dots,a_{n-1}$ and $b_0,\dots,b_{n-1}$ in the case $n=11$.}
\label{fig:ab}
\end{figure}

\begin{ai}
I first encountered the problem of determining the circular chromatic number of the real orthogonality graph on Open Problem Garden~\cite{OPG} in 2009 when I was an undergraduate student. I then learned more from reading~\cite{Ghebleh07} and~\cite{DeVos+09+}; I thank Luis Goddyn for sharing a draft of~\cite{DeVos+09+} with me. 

On September 11, 2026, I asked ChatGPT 6 Astra Ultra to prove~\cite[Conjecture~4.41]{Ghebleh07} and provided no additional context. In its first attempt, it was able to improve the known lower bound on $\chi_c(\mathcal{O})$ from $11/3$ to $23/6$ by combining a computation on the subgraph induced by normalizations of non-zero integer-valued vectors with entries between $-3$ and $3$ with a structural argument about the full orthogonality graph. The idea of using finite subgraphs of this type is not new. In an unpublished manuscript, DeVos, Ghebleh, Goddyn, Mohar and Naserasr~\cite{DeVos+09+} proved a lower bound of $7/2$ using the analogous subgraph for entries between $-1$ and $1$, and Ghebleh~\cite{Ghebleh07} proved a lower bound of $11/3$ by considering entries between $-2$ and $2$. 

In a second prompt in the same conversation, I provided ChatGPT 6 Astra Ultra with a picture of a $5$-cycle gadget that I had identified back in 2010, with a description of how I wanted to use it to restrict colour choices, and asked whether a generalization of it could be helpful. It then used a generalization of the gadget to produce a full proof of the conjecture, which is essentially the same as the proof that I have included here. I modified the presentation and rewrote it from scratch. ChatGPT assisted during the writing by proofreading, suggesting references, and answering questions about the arguments. Also, ChatGPT generated the diagram in Figure~\ref{fig:ab}. I take full responsibility for the contents of this paper and correctness of its arguments. 
\end{ai}


\begin{thebibliography}{10}

\bibitem{CherkashinVoronov24}
D.~Cherkashin and V.~Voronov.
\newblock On the chromatic number of 2-dimensional spheres.
\newblock {\em Discrete Comput. Geom.}, 71(2):467--479, 2024.

\bibitem{deBoerBuysZuiddam24}
D.~de~Boer, P.~Buys, and J.~Zuiddam.
\newblock The asymptotic spectrum distance, graph limits, and the {Shannon}
  capacity.
\newblock arXiv:2404.16763, 2024.

\bibitem{DeCortePikhurko16}
E.~DeCorte and O.~Pikhurko.
\newblock Spherical sets avoiding a prescribed set of angles.
\newblock {\em Int. Math. Res. Not. IMRN}, (20):6095--6117, 2016.

\bibitem{OPG}
M.~DeVos.
\newblock Circular colouring the orthogonality graph.
\newblock Open Problem Garden, 2008.
\newblock \url{https://www.openproblemgarden.org/op/circular_colouring_the_orthogonality_graph}.

\bibitem{DeVos+09+}
M.~DeVos, M.~Ghebleh, L.~Goddyn, B.~Mohar, and R.~Naserasr.
\newblock Colouring the real orthogonality graph.
\newblock Unpublished manuscript, 2009.

\bibitem{Ghebleh07}
M.~Ghebleh.
\newblock {\em Theorems and computations in circular colourings of graphs}.
\newblock PhD thesis, Simon Fraser University, Burnaby, British Columbia, Canada,
  2007.

\bibitem{GodsilZaks88}
C.~D. Godsil and J.~Zaks.
\newblock Colouring the sphere.
\newblock Research Report CORR 88-12, University of Waterloo, 1988.
\newblock See arXiv:1201.0486.

\bibitem{KochenSpecker67}
S.~Kochen and E.~P. Specker.
\newblock The problem of hidden variables in quantum mechanics.
\newblock {\em Journal of Mathematics and Mechanics}, 17(1):59--87, 1967.

\bibitem{Lovasz79}
L.~Lov\'asz.
\newblock On the {S}hannon capacity of a graph.
\newblock {\em IEEE Trans. Inform. Theory}, 25(1):1--7, 1979.

\bibitem{Rudin76}
W.~Rudin.
\newblock {\em Principles of Mathematical Analysis}.
\newblock McGraw-Hill, New York, third edition, 1976.

\bibitem{Vince88}
A.~Vince.
\newblock Star chromatic number.
\newblock {\em J. Graph Theory}, 12(4):551--559, 1988.

\bibitem{Viro+08}
O.~Ya. Viro, O.~A. Ivanov, N.~Yu. Netsvetaev, and V.~M. Kharlamov.
\newblock {\em Elementary topology}.
\newblock American Mathematical Society, Providence, RI, 2008.
\newblock Problem textbook.

\bibitem{Zhu92}
X.~Zhu.
\newblock Star chromatic numbers and products of graphs.
\newblock {\em J. Graph Theory}, 16(6):557--569, 1992.

\bibitem{Zhu01}
X.~Zhu.
\newblock Circular chromatic number: a survey.
\newblock {\em Discrete Math.}, 229(1--3):371--410, 2001.

\end{thebibliography}
\end{document}